\documentclass{article}
\usepackage{amsmath,dsfont, amssymb, amsthm}
\usepackage{color, hyperref}
\usepackage{graphicx,psfrag,epsf}
\usepackage{booktabs}
\usepackage{amsmath}
\usepackage{graphicx,psfrag,epsf}
\usepackage{enumerate}
\usepackage{natbib}
\usepackage[normalem]{ulem}

\newtheorem{proposition}{Proposition}
\newtheorem*{proposition*}{Proposition}

\theoremstyle{remark}

\newtheorem*{remark*}{Remark}
\newcommand\diag{{\rm diag}}

\title{An Orthogonally Equivariant  Estimator of the  Covariance Matrix in High Dimensions and for Small Sample Sizes }
\author{Samprit Banerjee\\
Weill Medical College of Cornell University\\
sab2028@med.cornell.edu
\and
Stefano Monni\\
American University of Beirut\\
sm150@aub.edu.lb
}

\begin{document}
\maketitle

\def\spacingset#1{\renewcommand{\baselinestretch}%
{#1}\small\normalsize} \spacingset{1}

\begin{abstract}
We introduce an estimation method of covariance matrices in a high-dimensional setting, i.e., when the dimension of the matrix, , is larger than the sample size . Specifically, we propose an orthogonally equivariant estimator. The eigenvectors of such estimator are the same as those of the sample covariance matrix. The eigenvalue estimates are obtained from an adjusted profile likelihood function derived by approximating the integral of the density function of the sample covariance matrix over its eigenvectors, which is a challenging problem in its own right. Exact solutions to the approximate likelihood equations are obtained and employed to construct estimates that involve a tuning parameter. Bootstrap and cross-validation based algorithms are proposed to choose this tuning parameter under various loss functions. Finally, comparisons with two well-known orthogonally equivariant estimators of the covariance matrix are given, which are based on Monte-Carlo risk estimates for simulated data and misclassification errors in real data analyses. In addition, Monte-Carlo risk estimates are also provided to compare our estimates of eigenvalues to those of a consistent estimator of population eigenvalues.
\end{abstract}

\spacingset{2}

\section{Introduction}
Many multivariate methods require an estimate of the covariance matrix. In this paper, we are interested in the problem of estimating the covariance matrix of  a multivariate normal distribution,     ${\mathcal N}(0, \Sigma)$,   using a sample  of mutually independent draws  $X_1, \ldots, X_n$,  from  it,   when  $n$ is less than the dimension $p$ of $\Sigma$.  This problem has received much attention in the recent past  because of an increasing number of applications  where  measurements are collected  on a large  number of variables, often  greater than the available  experimental units. The  sample covariance matrix  is not a good estimator in this case. In the general framework where both $p$ and $n$ go to infinity in such a way that their ratio $p/n$ converges  to a positive finite constant (often referred to as the large-dimensional asymptotic regime), the sample covariance matrix, its eigenvalues and its eigenvectors cease to be consistent.
Some  alternative estimators have thus been proposed  in the literature. \cite{LWSpectrum} propose estimators of the eigenvalues of the covariance matrix in the large-dimensional  framework that are consistent,  in the sense that the mean squared deviation between the estimated eigenvalues and the population eigenvalues converges to zero almost surely. Their method is based on  a particular discretization  of a version of the  Mar\v{c}enko-Pastur equation that links the limiting spectral distribution of the sample eigenvalues and that of the population eigenvalues \citep{LWQuest}.  This method is  then  used to derive estimators of the covariance matrix itself  that are asymptotically optimal with respect to a given loss function in the   space of  orthogonally equivariant estimators \citep{LWbernoulli}.  Additional results by these same authors have very recently appeared when our paper was under review  \citep{LW2020}.  Estimators  that are derived in the large-dimensional asymptotic regime are proposed  also by  \cite{elkaroui}, \cite{mestre}, \cite{yaokammounnajim}, among others.
Estimators that deal with the case $p>n$ and are  derived in a decision-theoretic framework are those of  \cite{konno}, and, more recently, \cite{tsukuma}. There is a vast literature on estimation of $\Sigma$   where  structural assumptions on $\Sigma$ are made such as ordering or  sparsity,  for example   \cite{BL.band, bien-tibs,  naul_taylor, won_bala}.

In this paper, we propose  an estimator for the covariance matrix  that is equivariant under  orthogonal transformations.  In particular,  these transformations include rotations of the variable axes.    Equivariant estimation of the covariance matrix under the orthogonal group has  been studied extensively (e.g.,  \cite{DeySri, LWSpectrum, takemura}) since the  pioneering work of \cite{Stein_Rietz, Stein_Russian2}.  In this study, we follow our previous work \cite{imple}, where we describe estimators that are  valid when $n>p$,   and extend it to the case when $p>n$.  Because of the property of equivariance, the eigenvectors of the covariance matrix are estimated by those of $S=\sum_{i=1}^nX_iX_i^\top$, to which we refer  as the sample covariance matrix in this paper.  Thus, the estimation problem is completed by providing estimates of the eigenvalues. These estimates are obtained from  an adjusted  profile likelihood function  of the population eigenvalues, which is derived by approximating the integral  of the density function of $S$ over its  eigenvectors (corresponding to the non-zero eigenvalues).  This approximation is however not the large-$n$ (Laplace) approximation of such integral, which results in the modified profile likelihood of \cite{barnd},  but it is an approximation  suggested in \cite{hikami-brezin} useful for large $p$.  Our estimates  are  a mixture $\hat{\lambda}^\kappa= (1-\kappa)\hat{\lambda}^{0}+\kappa \hat{\lambda}^{1}$  of an exact critical point $\hat{\lambda}^{0}$ of such a  likelihood function, which is in fact a maximum when some conditions are satisfied, and an approximate critical point $\hat{\lambda}^{1}$ whose components are  a modification of the non-zero sample eigenvalues by terms of order $1/p$.  The tuning parameter $\kappa$  is determined from the data and  controls the shrinkage of the eigenvalues. We will describe two algorithms  to determine $\kappa$: one based on bootstrap  and one based on  cross-validation.  High-dimensional estimators are generally derived under an asymptotic regime in which both   $n$  and $p$ increase in such a way that  their ratio  tends to a constant. In our case, $n$ is kept  fixed, and  the high-dimensionality of the estimator comes into play  because we consider a  large-$p$  approximation of the marginal density of the eigenvalues of $S$.

 In a variety of finite-sample simulation scenarios, we compare our estimator to two Ledoit-Wolf estimators, which are also orthogonally equivariant and  have previously  been shown to better many other estimators under some loss functions. Figures \ref{fig:PRIAL_LW1} and \ref{fig:PRIAL_LW2} summarize the results of our comparison, in terms of risk evaluated with respect to  nine loss functions.  To assess how our  covariance matrix estimator estimates the eigenvalues of the population covariance matrix,  we also compare the eigenvalues of our estimator  with   Ledoit-Wolf consistent estimator of the population eigenvalues themselves,  under  loss functions that depend only on the eigenvalues. Figure \ref{fig:PRIAL_LW3} displays such a comparison. Finally,   comparison of covariance matrix estimators is also carried out using two real data applications of breast cancer data and leukemia data: in a linear discriminant analysis of these data, we use plugged-in estimates of the covariance matrix in the classifier and demonstrate that our estimator leads to lower misclassification errors in the breast cancer data and similar misclassification errors in the leukemia data.  The two Ledoit-Wolf covariance matrix estimators are optimal asymptotically under two loss functions, but we show that finite sample improvements and improvements under other loss functions are indeed possible. Since the tuning parameter $\kappa$ of our estimator can be chosen by minimizing risk estimates with respect to any loss function, our estimator can be used with any loss function appropriate to a statistical application.

   This paper is organized as follows. In Section \ref{sec:generalities}, we introduce the adjusted profile likelihood that is used to obtain our estimator, which is introduced  and discussed in Section \ref{sec:estimator}.  Section \ref{sec:simulation}   and Section \ref{sec:LDA}  present  some numeric assessment of the performance of our estimator in  simulated  and real data respectively.

\section{Marginal Density and Likelihood Function}\label{sec:generalities}
In this section, we introduce some notations,  review the singular Wishart distribution, derive an approximation to the marginal density of the sample eigenvalues and then obtain an adjusted profile likelihood for the eigenvalues of the population covariance matrix. Consider the case of mutually independent  draws $X_1, \ldots, X_n$ from a multivariate $p$-dimensional normal distribution $\mathcal{N}(0, \Sigma)$,  with $\Sigma$ a $p \times p$ maximum-rank  positive definite matrix.  Assume $p>n$, and let $S$  be  the $p\times p$ sample covariance matrix   $X^\top X$,  with $X$ the matrix whose rows are the vectors $X_i^\top$.  $S$  is positive semi-definite and of  maximum rank, with  distinct positive eigenvalues: $\ell_1> \ell_2> \ldots> \ell_n>0$.  Geometrically, $S$  is an interior point of an $(n+1)n/2$-dimensional face of  the closed convex cone of semi-positive definite $p\times p$ symmetric matrices \citep{Barvinok}. \cite{uhlig} showed that  $S$  has  a distribution specified by the density
\begin{equation}\label{uhlig:density}
p(S)(dS) =K \left({\rm det} \Sigma \right)^{-n/2} {\rm etr}\left(-\Sigma^{-1}S/2 \right)\left({\rm det} L \right)^{(n-p-1)/2} (dS),
\end{equation}
where $K=\pi^{(-pn+n^2)/2} 2^{-pn/2}/\Gamma_n(n/2)$, $L=\diag(\ell_1,\ldots,\ell_n)$ is the diagonal matrix of the non-zero eigenvalues of $S$, ${\rm etr}(.)={\rm exp}({\rm tr}(.))$, and $(dS)$ is the volume element on the space of positive semi-definite $p \times p$ symmetric matrices of rank $n$,  with $n$ distinct positive eigenvalues. This distribution, which  extends the usual  ($n>p$)  Wishart distribution, is often called (non-central) singular Wishart distribution, but some authors \citep{SK_MVA}  prefer the name  non-singular pseudo-Wishart distribution. It corresponds to the case 7 of the classification scheme  of \cite{DG_GJ_Mardia} (described in Table 1 therein), where  generalizations are considered that include the cases when $\Sigma$ and  $S$ have non-maximum rank and when the samples are not independent.   The method  we will present can be also extended  to the  case when $S$ does not have maximum rank.  For example, in some applications, one may wish to center the observations to have mean zero. The  resulting matrix $S$ constructed from the centered data  will  have rank less than $n$.

Consider now the singular value decomposition of $X=UL^{1/2}H_1^\top$  with $U \in O(n)$ an orthogonal $n \times n$ matrix,   $L=\diag(\ell_1, \ldots, \ell_n)$   as defined above  and   $H_1$ the matrix whose $n$ columns are  the corresponding $n$  eigenvectors of $S$.   These $n$ eigenvectors are  uniquely determined up to column multiplication by $\pm 1$. The formulae below assume that one of these $2^n$ choices has been made. $H_1$ is a point in  the Stiefel manifold, $V_n({\mathbb R}^p)$, of all orthonormal  $n$-frames in ${\mathbb R}^p$. The joint density of $H_1$ and $L$ is
\[
p(H_1, L) \, (H_1^\top dH_1) \wedge (dL)
\]
where
\[
p(H_1, L) = 2^{-n} K \left({\rm det} \Sigma \right)^{-n/2} {\rm etr}\left(-\Sigma^{-1} H_1LH_1^\top/2 \right)\left({\rm det} L \right)^{(p-n-1)/2} \prod_{i<j}^n(\ell_i-\ell_j),
\]
$\wedge$ is the exterior product, $(dL)= d\ell_1 \wedge\cdots \wedge d\ell_n$,  and  $(H_1^\top dH_1)$ is  the  Haar measure of $V_n({\mathbb R}^p)$ normalized  as follows
\[
\int_{V_n({\mathbb R}^p)} (H_1^\top dH_1)= {\rm Vol}(V_n({\mathbb R}^p)) = \frac{2^n \pi^{pn/2}}{\Gamma_n(\frac{1}{2}p)}.
\]
Even if  we have written  the  densities of $S$ and of $H_1, L$  as differential forms, we  need   not keep track of the  sign of the  form, if we define  the integrals  to be positive. In the following, when the measure is clear, we will revert to using the term  density for  the scalar  part.

We are interested in obtaining an estimator for  $\Sigma$ that belongs to  the class of  orthogonally equivariant estimators. This equivariance is intended in the usual meaning.  Namely,  consider the action of the orthogonal group $O(p)$ on the sample space  that is defined by  $X \mapsto X G^\top$, or equivalently, by  $S \mapsto GSG^\top$.  We require   $\hat{\Sigma}(S) \mapsto G \hat{\Sigma}(S) G^\top$, which is the same way as $\Sigma$   transforms $\Sigma \mapsto G\Sigma G^\top$.   Under such action,  $p(S)$ is invariant and so is  the measure $(dS)$. This equivariance implies that the eigenvectors of  $\hat{\Sigma}$  are the same as those of $S$.  This is to say that the estimators are of the form  $\hat{\Sigma} =H \hat{\Lambda} H^\top$, where the elements of the  diagonal matrix  $\hat{\Lambda}$ are functions of the non-zero eigenvalues  of $S$  and   the orthogonal matrix $H =[H_1: H_2]$  is that of the eigenvectors of $S$ \citep{Stein_Russian2}.  In the case  in which $p>n$,  the zero eigenvalue has multiplicity $(p-n)$, so that the  corresponding eigenvectors of $S$   given by the $(p-n)$ columns $H_2=(h_{n+1}, \ldots, h_p)$,   are unique only up to an orthogonal transformation of the  last $(p-n)$ coordinate axes. Namely, we can consider $H_2$ or  $H_2\cdot P$ with any orthogonal matrix $P\in O(p-n)$. In general each different choice (of $P$) will lead to a different estimator of $\Sigma$. However,  if the estimates of the smallest $(p-n)$ eigenvalues of ${\Sigma}$   are  identical,  all such choices will be immaterial, in the sense that they will lead to the same $\hat{\Sigma}$. The estimates of the last $(p-n)$ eigenvalues that we propose   are indeed identical, so that one can use as estimates of all the eigenvectors of $\Sigma$  whatever representative  of the class of the eigenvectors of $S$ a numerical routine outputs.

To find the eigenvalue estimates,  we follow our previous paper \cite{imple}, and consider the marginal density of the sample roots of $S=X^\top X =H_1 L H_1^\top$:
\[
p(\ell_1, \ldots, \ell_n)=2^{-np/2} \frac{\pi^{(-pn+n^2)/2}}{2^{n}\Gamma_n(n/2)}  \left({\rm det} \Sigma \right)^{-n/2}
 \prod_{i=1}^n \ell_i^{(p-n-1)/2} \prod_{i<j}^n(\ell_i-\ell_j) \cdot J_n
\]
where
\[
J_n=\int_{V_n({\mathbb R}^p)} {\rm etr}\left(-\frac{1}{2}\Sigma^{-1} H_1 L H_1^\top \right) \,  (H_1^\top dH_1).
\]
The integral $J_n$ cannot  be computed in closed form. However,  in Appendix A, we derive a useful approximation. Namely,
\begin{proposition}\label{prop:eq:largep}
For large $p$, the integral $J_n$ is approximated by the following expression
\[
J_n  \approx   \exp\left(-\frac{1}{2}\sum_{i=1}^n \frac{ \hat{\ell}_i}{\lambda_i} \right) \cdot \prod_{i<j}^p \left( 1 + \frac{1}{p} \left(\frac{\lambda_i -\lambda_j}{\lambda_i \lambda_j}\right)(\hat{\ell}_i - \hat{\ell}_j)\right)^{-\frac{1}{2}}
\]
with $\hat{\ell}_i=\ell_i$ for $1 \leq i \leq n$ and $\hat {\ell}_i=0$ for $i>n$.
\end{proposition}

\begin{remark*}
The  following condition  $0 \leq (l_i-l_j) (\lambda_i -\lambda_j)/\lambda_i\lambda_j< p$ is employed in the proof of the proposition.
\end{remark*}

\begin{remark*}
This approximate formula is valid even when some of the population eigenvalues are equal. It is also interesting to notice that when $\Sigma=\mu I_p$ is proportional to the identity matrix, the   value of $J_n$ obtained using such formula is equal to the exact value of the integral, which can be computed to be $\exp(-\sum_{i=1}^n  \ell_i/2\mu)$ times the volume of the Stiefel manifold,  which is omitted in the  formula above (see Appendix A). In the following, however, we assume the population eigenvalues to be distinct  except when stated otherwise.
\end{remark*}

Employing such an approximation, we then obtain an approximate log-likelihood  function  for the true eigenvalues $\lambda$
\begin{eqnarray}\label{eq:loglik2}
{\mathcal L}(\lambda)  &=& -\frac{n}{2} \sum_{i=1}^p \ln {\lambda_i} -\frac{1}{2} \sum_{a=1}^n\frac{\ell_a}{\lambda_a}  \nonumber\\
 &&  -\frac{1}{2} \sum_{1\leq a<b \leq n} \ln \left( 1 + \frac{1}{p} \left(\frac{\lambda_a -\lambda_b}{\lambda_a \lambda_b}\right)(\ell_a - \ell_b)\right)\\
&&  -\frac{1}{2} \sum_{1 \leq a \leq n<r \leq p } \ln \left( 1 + \frac{1}{p} \left(\frac{\lambda_a -\lambda_r}{\lambda_a \lambda_r}\right)\ell_a\right) \nonumber.
\end{eqnarray}

The first two terms of this function are the profile log-likelihood function for the parameters $\lambda$,  which is the partially maximized log-likelihood function of $(\lambda, V)$,   where $V$ is replaced by the maximum likelihood estimator $\hat{V}_{\lambda}$ for fixed $\lambda$. We show in Appendix C that $\hat{V}_{\lambda}$ is  a solution to  the equation $V^\top H_1=M$, where $M$ is a $p \times n$ matrix such that $M_{ij}=\pm \delta_{ij}$, with $\delta_{ij}$    the Kronecker delta.  The other terms in ${\mathcal L}$  can thus be interpreted as an adjustment to the profile log-likelihood.

\section{The Proposed Estimator}\label{sec:estimator}
In this section, we derive an estimator for the eigenvalues of the population covariance matrix and discuss its properties.

Our starting point is ${\mathcal L}(\lambda)$, the function given in (\ref{eq:loglik2}), which can be  considered as a pseudo-(log)likelihood, of which our goal is to find the maximum points. We note that  this function  of $\lambda \in \mathbb{R}_{++}^p$  is not  concave   on the whole  domain $\mathbb{R}_{++}^p$, for all given values of $(\ell_1, \ldots, \ell_n)$.   The critical points are  the solutions to  the following equations
\begin{eqnarray}\label{eq:mle}
n \lambda_i   & =&    \ell_i-\frac{1}{p}  \sum_{b=1}^n\frac{\ell_i - \ell_b}{   1 + \frac{1}{p} (\frac{\lambda_i -\lambda_b}{\lambda_i \lambda_b})(\ell_i - \ell_b)     }
-\frac{\ell_i}{p}\sum_{r=n+1}^p \frac{1}{   1 + \frac{1}{p} (\frac{\lambda_i -\lambda_r}{\lambda_i \lambda_r})\ell_i   }
\end{eqnarray}
where $\ell_i=0$ when $i=n+1, \ldots, p$. Exact solutions to (\ref{eq:mle}) satisfy what we can call a trace condition: $n\sum_{i=1}^p \lambda_i =\sum_{a=1}^n \ell_a$, which is desirable since  $E(S)= n\Sigma$.

A solution to (\ref{eq:mle}) is seen to be given by    $\hat{\lambda}_i^0=\sum_{a=1}^n\ell_a/np$, for $i=1, \ldots, p$.  This    results in a diagonal estimator  for the covariance matrix $\hat{\Sigma}^0 = {\rm tr}(S) \cdot I_p /np$. However there is no guarantee  that such a solution
may be  a maximum of ${\mathcal L}$   for all given values of $(\ell_1, \ldots, \ell_n)$. Indeed, for some data (sample eigenvalues)   the Hessian matrix evaluated at this solution  can  have positive eigenvalues.  However, one can notice  that $\hat{\lambda}^0$  is a genuine maximum of the likelihood function  at order $1$  in $1/p$.    Namely, it is a  maximum of
\[
{\mathcal L}' =-n \sum_{i=1}^p \ln {\lambda_i} -\sum_{a=1}^n\frac{\ell_a}{\lambda_a}  -\frac{1}{p}\sum_{1\leq a<b \leq n} \frac{\lambda_a -\lambda_b}{\lambda_a \lambda_b}(\ell_a - \ell_b) -\frac{1}{p}\sum_{1 \leq a \leq n<r \leq p } \frac{\lambda_a -\lambda_r}{\lambda_a \lambda_r}\ell_a,
\]
which is obtained  by expanding to first order   the logarithm in (\ref{eq:loglik2}). Furthermore, $\hat{\lambda}^0$  is also the  maximum of the exact likelihood function  when  the true covariance matrix  is  proportional to the identity matrix, which can be computed
exactly,  without the need of any approximation, as remarked earlier.   The advantage of the solution $\hat{\lambda}^0$ is  that it shrinks the highest  and pushes up the lowest eigenvalue. Indeed,  the general theorem of \cite{vandervaart} tells us that the highest eigenvalue $\ell_i/n$ is upward biased and (obviously in this $p>n$ case) the lowest eigenvalue downward biased.  The disadvantage is that the shrinkage may be too extreme.   It is perhaps not surprising that the eigenvalue estimates are degenerate. In our derivation, the sample size $n$ is fixed and $p$ is  large, thus there may not be enough information in the sample to obtain a different estimate for each eigenvalue. To deal with the degeneracy of the estimates,  we construct  approximate solutions to the equations (\ref{eq:mle}).  We look for   $\hat{\lambda}_i$ with an expansion of the form $a_i +f_i/p +O(p^{-2})$, for $i=1, \ldots, n$.  There is no  reason a priori why   solutions should take this form. However, the  solution  $\hat{\lambda}_i^0$ is of this form, with $a_i=0$ and $f_i$  the same for all  $i=1, \ldots, n$.  Our goal is to  perturb the exact solution $\hat{\lambda}^0$ away from  having all components  equal, keeping the resulting estimates ordered and satisfying the trace condition. We  find
\begin{eqnarray*}\label{eq:appxestimate}
\hat{\lambda}^1_a   & =&   \frac{ \ell_a}{n}- \frac{1}{n} \sum_{b=1}^n \frac{\ell_a - \ell_b}{  p +n \left(\frac{1}{\ell_b} -\frac{1}{\ell_a}\right)(\ell_a - \ell_b) }, \quad \quad  a=1, \ldots, n \nonumber \\
\hat{\lambda}^1_r &=&  0,  {\hskip 5.5cm} r=n+1, \ldots, p.
\end{eqnarray*}
which  are  a modification of the eigenvalues $\ell_i/n$ of the (usual) sample covariance matrix $S/n$ (in our conventions $\ell_i$ are the eigenvalues of $S=X^\top X$)
with a correction term of order $1/p$.   We do not use such an approximate solution as the estimate of our eigenvalues (it would lead to a non-invertible estimator of the covariance matrix, for one thing).  We employ  $\hat{\lambda}^1$ as a perturbation of the true solution $\hat{\lambda}^0$. In fact, what we  propose as  an estimator  is  a linear combination of  $\hat{\lambda}^0$ and $\hat{\lambda}^1$,  controlled by a tuning parameter $\kappa$. Namely,
\begin{eqnarray}\label{eq:bmestimate}
\hat{\lambda}^{\kappa}_a   & =&   \frac{\kappa}{n}  \left( \ell_a-  \sum_{b=1}^n \frac{\ell_a - \ell_b}{  p +n \left(\frac{1}{\ell_b} -\frac{1}{\ell_a}\right)(\ell_a - \ell_b) }\right) + (1-\kappa)  \frac{\sum_{a=1}^n \ell_a}{np},  \, a=1, \ldots, n \nonumber \\
\\
\hat{\lambda}^{\kappa}_r &=& (1-\kappa)  \frac{\sum_{a=1}^n \ell_a}{np},  {\hskip 5.5 cm} r=n+1, \ldots, p  \nonumber
\end{eqnarray}
where $0 \leq \kappa < 1$. The parameter $\kappa$  controls the  shrinkage of the eigenvalue estimate and is to be determined from the data. When $\kappa$ is zero, the shrinkage is highest, and we recover the solution to the ML  equations,  with all the eigenvalues being equal and, accordingly,   $\hat{\Sigma}=\hat{\Sigma}^0$ is  proportional to the identity $I_p$. When $\kappa$ tends to one we get $\hat{\lambda}^1$ with distinct estimates of the first $n$ eigenvalues.  There is no guarantee that $\hat{\lambda}^{\kappa}$ is a maximizer for (\ref{eq:loglik2}).  For these reason, one could   try to maximize  (\ref{eq:loglik2})  numerically.   We did consider numerical solutions to (\ref{eq:mle}) using Newton's method and a constraint of positivity on the solutions. The resulting roots    were always found to  be close to $\hat{\lambda}^1$ with  the last $(p-n)$ values  negligible.  Furthermore,  when used in place of $\hat{\lambda}^1$, in the estimator (\ref{eq:bmestimate}), these numerical components had similar estimates of risk compared to our estimator  (\ref{eq:bmestimate}) in the simulation study conducted in Section \ref{sec:SimStudy} (results not shown), but  added an un-necessary  computational step.

The  last $(p-n)$ estimates  of the eigenspectrum (\ref{eq:bmestimate})   are all equal. As  observed in Section \ref{sec:generalities}, this property guarantees that  any chosen basis for  the eigenspace  corresponding to the zero eigenvalues of  $S$  will  give rise to the same  estimator $\hat{\Sigma}^\kappa$. In fact,  this property should be required of any orthogonally equivariant estimators of $\Sigma$ in the $p>n$ setting, although it has not been explicitly mentioned before, and  it also holds for the non-linear estimators of \cite{LWSpectrum}. Our proposed estimators for the true eigenvalues  have the following additional properties proven in Appendix B.
\begin{proposition}[The  properties of the eigenvalue estimates] \label{prop:properties}
For $\kappa$, such that $0 \leq \kappa <1$, the estimates $\hat{\lambda}^\kappa$  given in (\ref{eq:bmestimate}) have the following properties
\begin{itemize}
\item[a)]{  $\hat{\lambda}^\kappa_1 > \hat{\lambda}^\kappa_2 >  \ldots > \hat{\lambda}^\kappa_n > \hat{\lambda}^\kappa_{n+1}= \ldots =\hat{\lambda}^\kappa_p >0$}
\item[b)]{n $\sum_{i=1}^p \hat{\lambda}^\kappa_i = \sum_{a=1}^n \ell_i$}
\end{itemize}
\end{proposition}
\noindent
Thus the corresponding estimator $\hat{\Sigma}_\kappa$ of $\Sigma$ will be positive definite. In addition, because of the ordering of the estimates, there is no additional step,  such as re-ordering or isotonization, that often is necessary.
The computational burden of obtaining the proposed estimates only stems from finding  the singular value decomposition of the data matrix $X$ or the eigenspectrum of $S$, and by the evaluation of the parameter $\kappa$, which we discuss in Section  \ref{sec:kappa}.

The formulae presented so far have been obtained under the assumption that the data matrix $X$ or  $S=X^\top X$ were of maximum rank $n$. In some applications one may wish to center the data, $Y_i=X_i-\bar{X}$, and consider the matrix $\sum_{i}Y_iY_i^\top$  in place of $S$.
All formulae can be applied to these situations, if we replace $n$ with the rank $q$ of the rescaled matrix in the corresponding maximum rank equations. A sketch of their derivation is given in Appendix D.

\subsection{Selecting the Tuning Parameter $\kappa$}
\label{sec:kappa}
The tuning  parameter $\kappa, \; 0 \leq \kappa < 1$,  of $\hat{\lambda}^{\kappa}$ needs to be determined from data. Selection of tuning parameters in an unsupervised setting is a  difficult problem, and there is no method which is always satisfactory. In the context of covariance estimation, tuning parameters are often determined by minimizing estimates of risks (see for example  \cite{BL.band} for a cross-validation approach and \cite{YiZou} for an approach using Stein's unbiased estimate of risk). This is also the approach we follow to choose $\kappa$, although our estimates of risk differ. Namely, we consider some loss function  $L(\hat{\Sigma},\Sigma)$ and compute  the corresponding risk as follows:
\begin{equation}\label{eq:risk}
R(\hat{\Sigma},\Sigma) = E(L(\hat{\Sigma},\Sigma)),
\end{equation}
 where  the expectation is over the data distribution. When $\hat{\Sigma}=\hat{\Sigma}^{\kappa}$,  the risk can be seen as a function of $\kappa$.  The ``oracle'' $\kappa$ is then  $\kappa' = \operatornamewithlimits{argmin}\limits_{\kappa} R(\kappa)$.

It is noteworthy that estimating the risk and estimating the tuning parameter that minimizes that risk, are not necessarily the same  problem.  The estimation of the risk  is complicated by  the fact that the true population matrix $\Sigma$ is unknown in practice.  We propose two methods  to estimate the risk of the estimator $\hat{\Sigma}^{\kappa}$ under a loss $L$: one method relies on a bootstrap re-sampling scheme and the other on cross-validation.

\subsubsection*{$\kappa$-selection via bootstrap}
To estimate the risk via bootstrap, we randomly choose $n$ rows with replacement from our data matrix $X$. Let $X_b$ be such a sample and $S_b = X_b^\top  X_b$ be  the  corresponding  sample covariance matrix. We then compute the reduced rank estimator, $\hat{\Sigma}^{\kappa}_{b}$, as described in Appendix D  and  evaluate  the loss  $L(\hat{\Sigma}^{\kappa}_{b}, \bar{\Sigma})$  with respect to a reference estimator $\bar{\Sigma}$ for a grid of values of $\kappa \in [0,1)$. This procedure is iterated $B$ times. The risk estimate is taken to be
\[
\hat{R}(\kappa)=\frac{1}{B} \sum_{b=1}^B L(\hat{\Sigma}^{\kappa}_{b}, \bar{\Sigma})
\]
and  the optimal $\kappa$ determined as
\[
\hat{\kappa} = \operatornamewithlimits{argmin}\limits_{\kappa \in [0,1)} \hat{R}(\kappa).
\]

The choice of the reference estimator  $\bar{\Sigma}$ in our proposed $\kappa$-selection procedure requires discussion.  We have considered using $S$ and $\hat{\Sigma}^1$ (matrix estimator corresponding to $\hat{\lambda}^1$) as reference estimators. However, since these estimators are singular, they may not be used when computing loss functions that require their inversions (such as  Stein's loss function or the quadratic loss function, see Section \ref{sec:simulation}). In these cases, we have flipped the role of the reference estimator and the estimator at hand  when computing the loss functions. As an alternative approach, we have used as reference estimator a non-singular extension of $\hat{\Sigma}^1$,  which we call  $\hat{\Sigma}^1_{NS}$,  where the zero eigenvalues are replaced with the smallest non-zero eigenvalue estimate. Our simulation studies  (not shown here)  indicate that the second strategy of using a non-singular estimator performs better than the first approach in selecting $\kappa$. We should emphasize that the choice of $\kappa$ closest to the $\kappa'$ depends largely on the reference estimator and a better reference estimator will lead to a much improved estimator.

\subsubsection*{$\kappa$-selection via cross-validation}
A second method  estimates the risk using  cross-validation (CV).  Different losses lead to different estimates, and we have implemented this method for Stein's, quadratic and Frobenius loss functions (see Section \ref{sec:LossF} for definition of these loss functions). First, consider the Frobenius loss, or actually its square, for reasons that we will be readily apparent:
\[
frob(\hat{\Sigma}^\kappa, \Sigma)^2 = {\rm tr}\left( (\hat{\Sigma}^\kappa-\Sigma)(\hat{\Sigma}^\kappa-\Sigma)^\top \right) = {\rm tr}(\hat{\Sigma}^\kappa \hat{\Sigma}^\kappa)  -2 {\rm tr} (\Sigma \hat{\Sigma}^\kappa)  +{\rm tr}(\Sigma^2).
\]
We can now ignore terms  that do not depend on $\kappa$ since we wish to minimize  with respect to $\kappa$. Observing further that
\[
{\rm tr} (\Sigma \hat{\Sigma}^\kappa)= E_*\left({\bf x}^\top_* \hat{\Sigma}^\kappa {\bf x}_*\right)
\]
where the expectation $E_*$ is taken with respect to the distribution of  ${\bf x}_*$, an independent sample  from ${\cal N}(0, \Sigma)$, we obtain the following leave-one-out  cross-validation estimate of the risk
\begin{eqnarray*}
\hat{{frob}}_{\kappa}^2=  {\rm tr}(\hat{\Sigma}^{\kappa} \hat{\Sigma}^{\kappa}) -\frac{2}{n}\sum_{i=1}^n  X^\top_i (\hat{\Sigma}_{\setminus i}^{\kappa})  X_i
\end{eqnarray*}
where $\hat{\Sigma}_{\setminus i}^\kappa$ is the estimator obtained removing the $i$-th row $X^\top_i$ from $X$. If $n$ is not very small, one can consider  $K$-fold CV instead of leave-one-out CV, to ease the computational burden.  When the loss function is  quadratic or Stein's, we reverse the role of $\Sigma$ and $\hat{\Sigma}$ and  consider $L(\Sigma, \hat{\Sigma}^\kappa)$  rather than $L(\hat{\Sigma}^\kappa, \Sigma)$.  We can express the traces involving $\Sigma$ and the inverse of $\hat{\Sigma}^\kappa$  as expectation with respect to samples of  $\mathcal{N}(0, \Sigma)$, and we are able to obtain the  estimates of risks, or, more accurately, of functions that have the same minimum as the risks, since terms that do not depend on $\kappa$ can be ignored. More precisely,  the quadratic loss with the reversed role of $\Sigma$ and $\hat{\Sigma}$ is
\begin{eqnarray*}
q(\Sigma, \hat{\Sigma})&=&{\rm tr} ((\Sigma \hat{\Sigma}^{-1} -I)^2)=p -2 {\rm tr}(\Sigma \hat{\Sigma}^{-1} ) +{\rm tr} (\Sigma \hat{\Sigma}^{-1}\Sigma \hat{\Sigma}^{-1})\\
&=& p-2 E_*\left({\bf x}^\top_* \hat{\Sigma}^{-1} {\bf x}_*\right)+\frac{1}{2} E_*\left({\bf x}^\top_* \hat{\Sigma}^{-1} {\bf x}_*  \cdot {\bf x}^\top _* \hat{\Sigma}^{-1} {\bf x}_*\right)\\
&&-\frac{1}{2} \left(E_*\left({\bf x}^\top_* \hat{\Sigma}^{-1} {\bf x}_*\right)\right)^2.
\end{eqnarray*}
Ignoring terms that do not depend on $\kappa$, the  leave-one-out cross-validation estimate of the risk under the  quadratic loss is
\[
\hat{q}_\kappa= -\frac{2}{n}\sum_{i=1}^n X^\top_i (\hat{\Sigma}_{\setminus i}^{\kappa})^{-1} X_i + \frac{1}{2n}\sum_{i=1}^n \left(X^\top_i (\hat{\Sigma}_{\setminus i}^{\kappa})^{-1} X_i \right)^2  -\frac{1}{2} \left( \frac{1}{n}\sum_{i=1}^n X^\top_i (\hat{\Sigma}_{\setminus i}^{\kappa})^{-1} X_i  \right)^2.
\]
The Stein's loss  $st(\Sigma, \hat{\Sigma})$, which is now  twice the  KL-divergence of normal densities with covariance matrix $\Sigma$ and $\hat{\Sigma}$, is equal, up to terms that do not depend   on $\kappa$, to
\begin{eqnarray*}
{\cal S}_\kappa &=& \frac{1}{2}E_*\left({\bf x}^\top_* \hat{\Sigma}^{-1} {\bf x}_*\right)  + \frac{1}{2} \ln \det  \hat{\Sigma}^{\kappa},
\end{eqnarray*}
from which one obtains the leave-one-out cross-validation   risk estimate
\begin{equation*}
\hat{st}_\kappa = \frac{1}{2n}\sum_{i=1}^n X^\top_i (\hat{\Sigma}_{\setminus i}^{\kappa})^{-1} X_i  + \frac{1}{2}\ln \det \hat{\Sigma}^{\kappa}.
\end{equation*}

\begin{figure}
\includegraphics[scale=1,page=1]{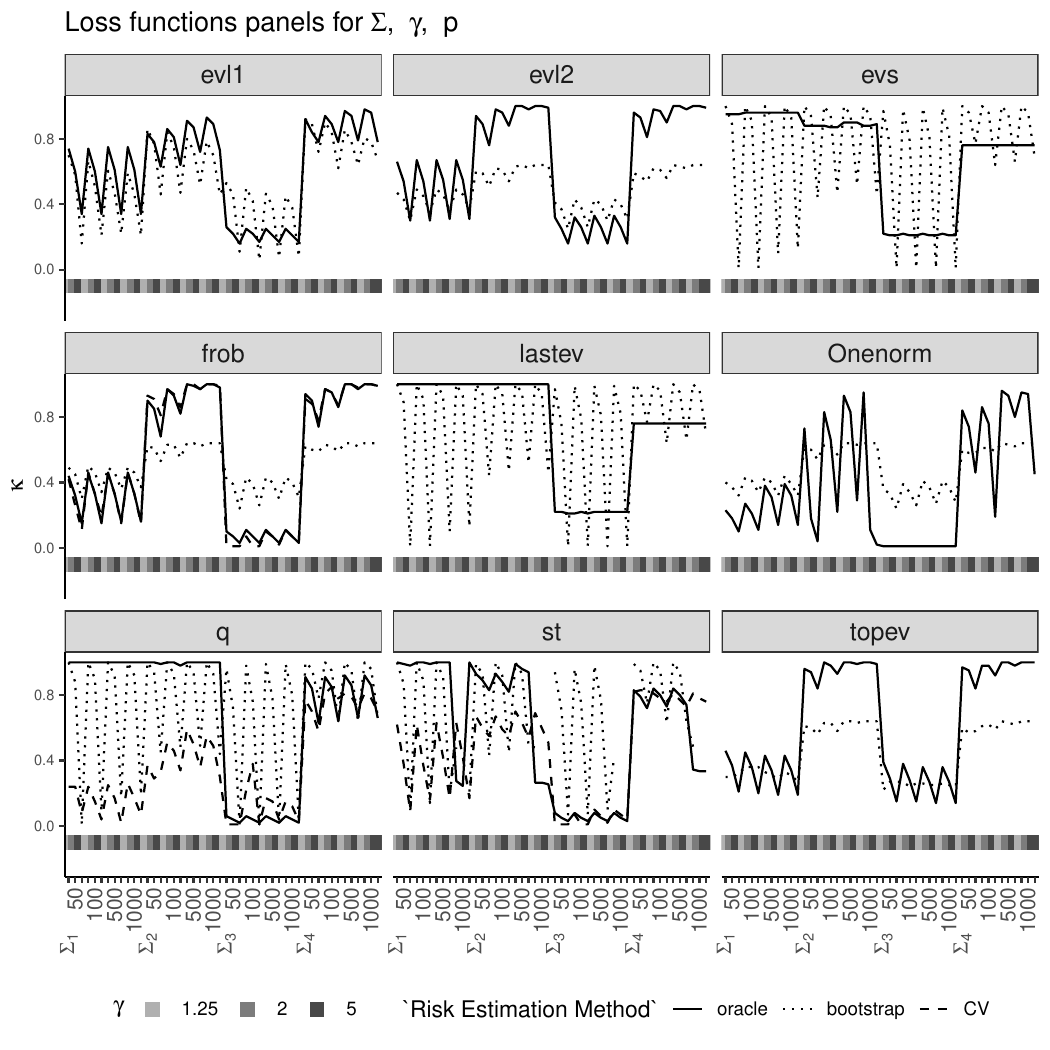}
\caption{\textbf{$\kappa$-selection for $\Sigma_1$ to $\Sigma_4$:} Chosen $\hat{\kappa}$ via bootstrap (dotted line) and CV (dashed line) are compared with oracle $\kappa'$ (solid line) with nine loss functions (each panel) for various combinations of true $\Sigma$ ($\Sigma_1$ to $\Sigma_4$), $\gamma=p/n$ and $p$ (see Section \ref{sec:simulation}).}
\label{fig:kappa1}
\end{figure}

\begin{figure}
\includegraphics[scale=1,page=2]{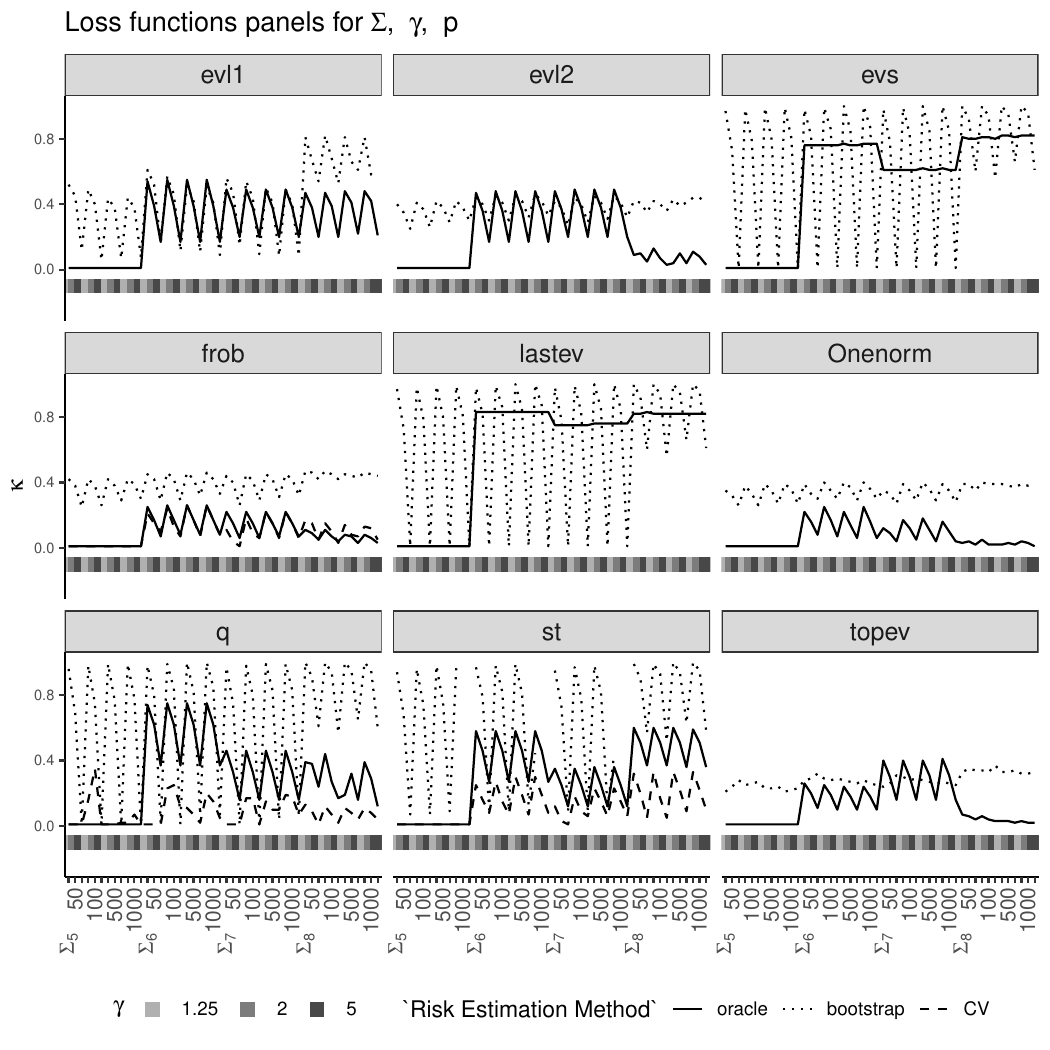}
\caption{\textbf{$\kappa$-selection for $\Sigma_5$ to $\Sigma_8$:} Chosen $\hat{\kappa}$ via bootstrap (dotted line) and CV (dashed line) are compared with oracle $\kappa'$ (solid line) with nine loss functions (each panel) for various combinations of true $\Sigma$ ($\Sigma_5$ to $\Sigma_8$), $\gamma=p/n$ and $p$ (see Section \ref{sec:simulation}).}
\label{fig:kappa2}
\end{figure}

\subsubsection*{Numerical comparison of $\kappa$-selecting methods}

We conducted a simulation study (see Section \ref{sec:SimStudy} for details) to evaluate these two strategies and compare the corresponding values of $\kappa$ with the ``oracle'' $\kappa'$. We considered nine different loss functions (see Section \ref{sec:LossF}), eight different covariance structures and $p=50, 100, 500, 1000$ with $\gamma = \frac{p}{n} = 1.25, 2, 5$. The results of this study are shown in Figures \ref{fig:kappa1} and \ref{fig:kappa2} where each panel corresponds to a loss function, the horizontal axis of each panel represents various combinations of the true covariance structure $\Sigma$, $\gamma$ and $p$,  and the vertical axis of each panel represents the chosen $\kappa$ for the two proposed methods (CV in dashed line and bootstrap in dotted line) and the ``oracle" $\kappa'$ (in solid line) which is determined using the true matrix $\Sigma$. Figure \ref{fig:kappa1} has true covariance structures $\Sigma_1$ to $\Sigma_4$ and Figure \ref{fig:kappa2} has true covariance structures $\Sigma_5$ to $\Sigma_8$. Note that the leave-one-out CV is only available for three out of the nine loss functions (i.e., Frobenius, Stein's and quadratic). The leave-one-out CV estimates the oracle $\kappa'$ almost perfectly for the Frobenius loss and also works well for Stein's and quadratic loss functions. The bootstrap method on the other hand works reasonably well in choosing $\kappa$ for most loss functions except the ones that depend on the smallest eigenvalues entirely. We have observed in our simulations that when the bootstrap estimate $\hat{\kappa}$ is quite different from the oracle $\kappa'$, the risk curves are quite flat,  i.e.,  $\hat{R}(\hat{\kappa})$ and $\hat{R}(\kappa')$ are quite similar and there is little risk improvement with different choices of $\kappa$.

As noted earlier, the choice of $\kappa$ depends largely on the reference estimator. When analyzing real data, we recommend the analyst to choose a loss function that is appropriate for the applied problem. If any of Stein's, quadratic or Frobenius (i.e., $st$, $q$ or $frob$) loss functions is an appropriate choice, we recommend using the CV-method to select $\kappa$ as it does not depend on a reference estimator. For any other loss function, we would recommend using other estimators as reference (e.g.,  the estimators of Ledoit and Wolf) in addition to $\hat{\Sigma}^1_{NS}$.

\section{Numerical Risk Comparisons with Other Estimators}\label{sec:simulation}
In this section we perform Monte Carlo simulations to evaluate our proposed estimator with respect to various loss functions.  Specifically, we   compare it  with the two non-linear shrinkage estimators of Ledoit and Wolf, one which is asymptotically optimal under Frobenius loss (LW1) and one which is  asymptotically optimal under Stein's loss (LW2). Ledoit and Wolf's nonlinear shrinkage  estimators  are of the form $H \hat{D} H^\top$, and thus orthogonally equivariant,  with $\hat{D}=\diag(\hat{\varphi}^*(\ell_1), \ldots \hat{\varphi}^*(\ell_p))$. The function $\hat{\varphi}^*$ is the nonlinear  function  responsible for shrinking the sample eigenvalues. Its form depends on which loss function is asymptotically minimized \citep{LWbernoulli}.  We refer the reader to   Section 3.2 of \cite{LWSpectrum} for  the specific form  of $\hat{\varphi}*$  when the loss is Frobenius and to Sections 5 and 6  of \cite{LWbernoulli} for its explicit form  when the loss is Stein's. In addition, we compare our eigenvalue estimates with  those of the Ledoit-Wolf consistent estimator of the population eigenvalues  \citep{LWSpectrum}, which we call LW3.  The comparison is carried out    under the loss functions  that only depend on eigenvalues:  loss functions 3, 4, 7, 8, and 9 in the Section \ref{sec:LossF} below.

\subsection{Loss Functions}\label{sec:LossF}
The comparison of the estimators  is carried  out  using  the Monte Carlo estimates of risk as defined in eq. (\ref{eq:risk}) with respect to the nine loss functions defined below.  We choose a variety of loss functions, some of which depend on the complete covariance matrix and its estimate, and others depend only on the eigenvalues and their estimates. To an applied statistician, the type of loss function is determined by the context of the application which could involve the estimation of eigenvalues, an important characteristic of the covariance matrix. Thus, the reader will get a sense of overall performance of an estimator under various loss functions. We also refer the reader to our previous paper \cite{imple} for a detailed description of most of these loss functions.

\begin{enumerate}
\item{Stein's (entropy) loss $st(\hat{\Sigma}, \Sigma) = {\rm tr} \,(\hat{\Sigma}\Sigma^{-1} -I) -\ln \det (\hat{\Sigma}\Sigma^{-1})$;}
\item{the quadratic loss $q(\hat{\Sigma}, \Sigma) ={\rm tr} \,(\hat{\Sigma}\Sigma^{-1} -I )^2$;  }
\item{$L_1$ eigenvalue loss $evl1(\hat{\Sigma}, \Sigma)=\sum_{i=1}^p|\hat{\lambda}_i-\lambda_i|/p$;}
\item{$L_2$ eigenvalue loss $evl2(\hat{\Sigma}, \Sigma)=\sum_{i=1}^p(\hat{\lambda}_i-\lambda_i)^2/p$;}
\item{Frobenius  loss  $frob(\hat{\Sigma}, \Sigma)=||\hat{\Sigma}-\Sigma||_F$, with $||A||_F^2 = {\rm tr}(AA^\top)$;}
\item{Matrix $L_1$-norm, the max of the $L_1$ norm of the columns of $|\hat{\Sigma}-\Sigma|$ or $||\hat{\Sigma}-\Sigma||_{1,1}, Onenorm(\hat{\Sigma},\Sigma)=
\operatornamewithlimits{max}\limits_{1 \leq j \leq p}  \sum\limits_{i=1}^{p} | \hat{\sigma}_{ij} - \sigma_{ij}|$;
}
\item{ $L_1$ loss on the largest eigenvalue $TopEV(\hat{\Sigma}, \Sigma)=|\hat{\lambda}_1-\lambda_1|$;}
\item{ $L_1$ loss on the smallest eigenvalue $LastEV(\hat{\Sigma}, \Sigma)=|\hat{\lambda}_p-\lambda_p|$;}
\item{ $L_1$ loss on the smallest quartile of the eigenvalues $EVS(\hat{\Sigma}, \Sigma)=\sum_{i=\lceil 3p/4 \rceil }^p|\hat{\lambda}_i-\lambda_i|$.}
\end{enumerate}

\noindent
We notice that, for our estimator $\hat{\Sigma}^{\kappa}$, the oracle $\kappa'$  is chosen as described in Section \ref{sec:kappa}.

\subsection{Simulation Study}\label{sec:SimStudy}
We construct eight covariance structures to represent typical applications. The matrix $\Sigma_1$ has widely spaced eigenvalues, $\Sigma_2$ has one large eigenvalue and mimics a typical principal components analysis covariance structure, $\Sigma_3$  is a time series example, $\Sigma_4$ is a spiked covariance structure, $\Sigma_5$ is the identity matrix, $\Sigma_6$ has eigenvalues drawn from a U-shaped beta distribution, $\Sigma_7$ has eigenvalues drawn from a linearly decreasing beta density and $\Sigma_8$ has eigenvalues of 1, 3 and 10 distributed with a frequency of 20\%, 40\% and 40\% respectively.  Namely,
 \begin{enumerate}
         \item $\Sigma_1 = \diag(p^2,\cdots,2^2,1^2)$;
        \item $\Sigma_2 = \diag(\lambda^{\ast}_{(1)},\lambda_{(2)},\cdots,\lambda_{(p)})$, where $\lambda_i \sim U(1,p/2)$, $\lambda_{(i)}$ are the ordered $\lambda_i$'s, $\lambda^{\ast}_{(1)}=\lambda_{(1)}^2$, with $U(a,b)$ being  the uniform distribution over the $[a,b]$ interval;
        \item $\Sigma_3 = AR(1)$, the first-order  autoregressive covariance matrix,  where $\sigma_{ij}=4 \times 0.7^{|i-j|}$ for $i \neq j$ and $\sigma_{ii}=4^2$ for $i=1,\cdots,p$;
        \item $\Sigma_4 = \diag(2p,p,1,\cdots,1)$;
        \item $\Sigma_5 = I_p$ where $I_p$ is the $p$-dimensional identity matrix;
        \item  $\Sigma_6 = \diag(\lambda_i)$ where $\lambda_i = 1 + 9F^{-1}_{(0.5,0.5)} \left( \frac{i}{p}-\frac{1}{2p} \right) \; i=1,\cdots,p$ and $F(\alpha,\beta)$ is the cumulative distribution function (c.d.f) of a beta distribution with parameters $(\alpha,\beta)$; the choice of $\alpha=0.5, \beta=0.5$ draws eigenvalues from a U-shaped density with more mass on large (around 10) and small (around 1) eigenvalues as the support of the beta density has been shifted to $[1,10]$;
        \item $\Sigma_7$ is similar to $\Sigma_6$ with $(\alpha=1,\beta=2)$ for the shape parameters of the beta c.d.f reflecting a linearly decreasing triangle with the highest density at $1$ and lowest density at $10$;
        \item $\Sigma_8$ has 20\% of its eigenvalues equal to 1, 40\% equal to 3 and the remaining 40\%  equal  to 10.

    \end{enumerate}

\noindent
For each covariance structure $\Sigma_j$ ($j=1,2,3,4,5,6,7,8$) four values of $p$ are considered $p=50,100,500,1000$.  For all eight covariance structures and their various dimensions, three values of $n$ are chosen corresponding to $p/n=\gamma=1.25,2,5$. We generate $n$ vectors $X_i \sim \mathcal{N}(0,\Sigma_j)$, $i=1,\cdots,n$,  for the first seven covariance structures. For the eighth case,  we  generate $n$ vectors from a $p$-variate Student $t$ distribution with four degrees of freedom so that $\Sigma_8$ is  used to test robustness to deviations from normality. We evaluate all nine loss functions, denoted henceforth by $st$, $q$, $evl1$, $evl2$, $frob$, $Onenorm$, $topev$, $lastev$ and $evs$, on our estimator and compare  the latter   with  LW1 (nonlinear shrinkage estimator that is optimal under Frobenius loss) and LW2 (nonlinear shrinkage estimator optimal under Stein's loss) (see Section \ref{sec:LossF}).  Additionally, we evaluate  the five loss functions ($evl1$, $evl2$, $topev$, $lastev$, $evs$)  that depend only on eigenvalues to compare our estimator with LW3 (the Ledoit-Wolf consistent estimator of population eigenvalues).  Risk estimates are based on 1000 repetitions for each simulation scenario.

\begin{figure}
\includegraphics[scale=1,page=4]{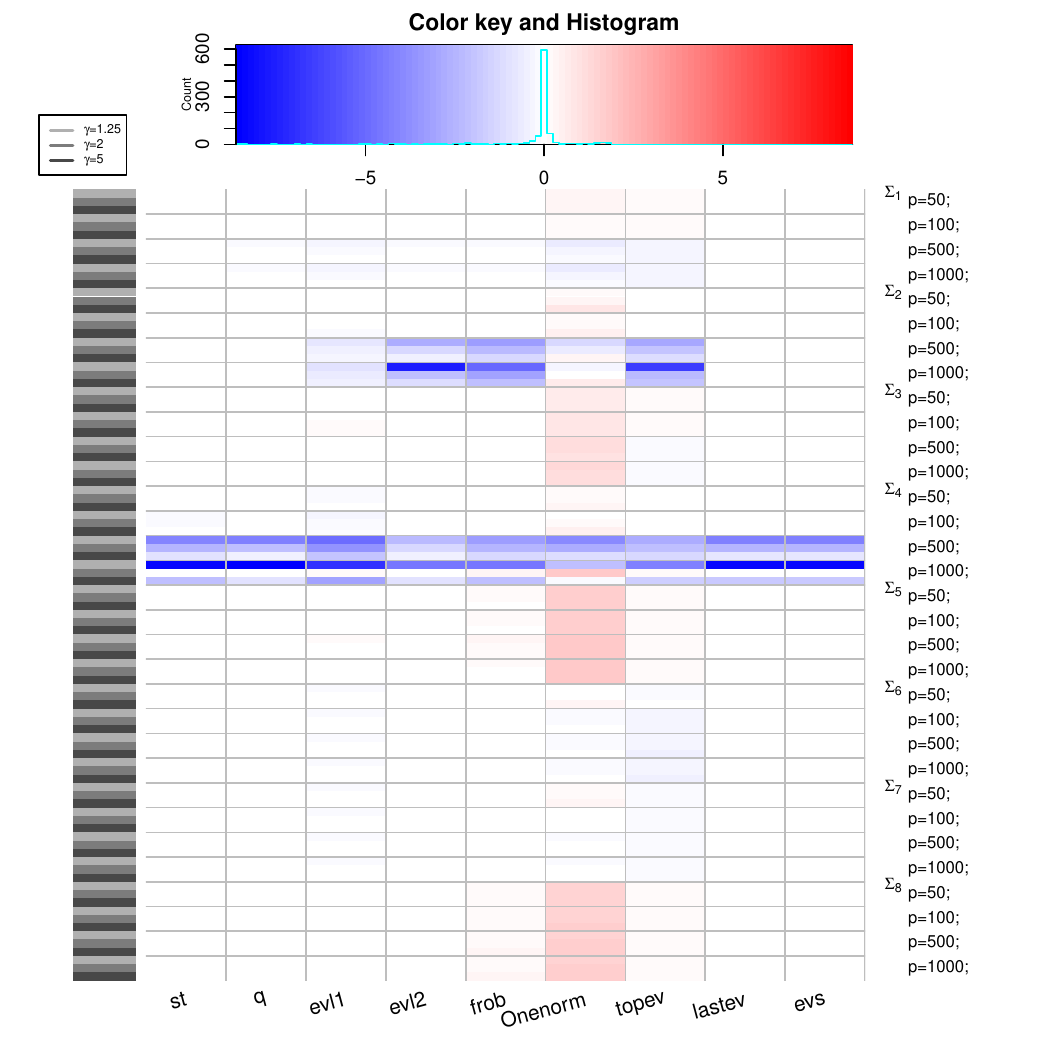}
\caption{\textbf{PRIAL comparison with LW1}. The heatmap shows values of PRIAL of $\hat{\Sigma}$ with respect to $\hat{\Sigma}^{LW_1}$ for various simulation scenarios based on covariance structure, values of $p$ and $\gamma$ (rows) and various loss functions (columns). Top panel shows counts and histogram of various PRIAL values. PRIAL values are scaled column-wise for visual clarity. Red (blue) shades mean $PRIAL>0$ ($PRIAL<0$) indicating our estimator is better (worse) than LW1.}
\label{fig:PRIAL_LW1}
\end{figure}

\begin{figure}
\includegraphics[scale=1,page=4]{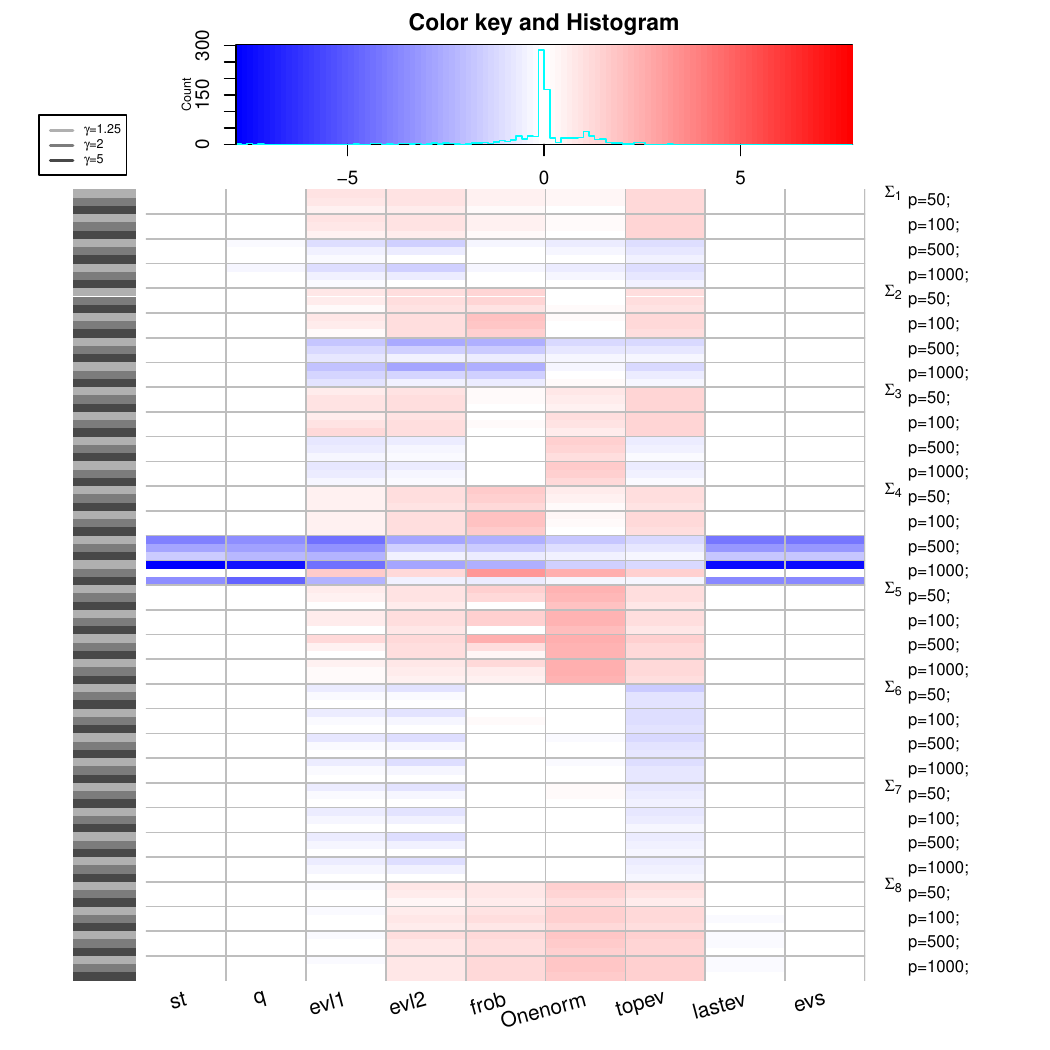}
\caption{\textbf{PRIAL comparison with LW2}. The heatmap shows values of PRIAL of $\hat{\Sigma}$ with respect to $\hat{\Sigma}^{LW_2}$ for various simulation scenarios based on covariance structure, values of $p$ and $\gamma$ (rows) and various loss functions (columns). Top panel shows counts and histogram of various PRIAL values. PRIAL values are scaled column-wise for visual clarity.  Red (blue) shades mean $PRIAL>0$ ($PRIAL<0$) indicating our estimator is better (worse) than LW2.}
\label{fig:PRIAL_LW2}
\end{figure}

\begin{figure}
\includegraphics[scale=1,page=4]{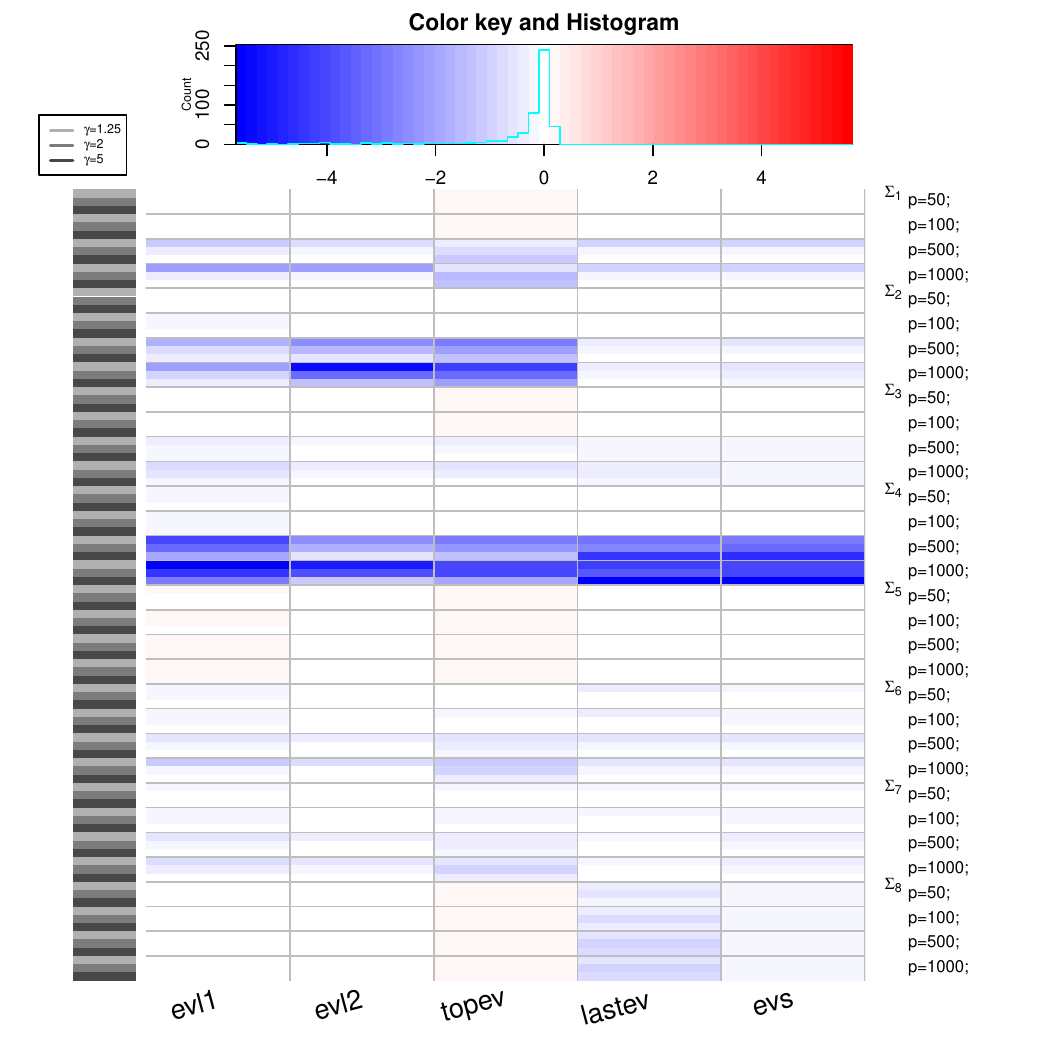}
\caption{\textbf{PRIAL comparison with LW3}. The heatmap shows values of PRIAL of $\hat{\lambda}$ with respect to $\hat{\lambda}^{LW_3}$ for various simulation scenarios based on covariance structure, values of $p$ and $\gamma$ (rows) and various loss functions (columns). Top panel shows counts and histogram of various PRIAL values. PRIAL values are scaled column-wise for visual clarity.  Red (blue) shades mean $PRIAL>0$ ($PRIAL<0$) indicating our estimator is better (worse) than LW3.}
\label{fig:PRIAL_LW3}
\end{figure}

As a measure of comparison, we use the Proportion Reduction in Integrated Average Loss  (PRIAL)  of  our estimator $\hat{\Sigma}$ over  $\hat{\Sigma}^{LW_k}$, which is LW1 ($k=1$) or LW2 ($k=2$).   Namely,  for a loss function  $L(.)$, PRIAL is defined similarly to    that in \cite{LiuPerl}:
\begin{equation}\label{eq:PRIAL}
 \left( \sum_i L(\hat{\Sigma}^{LW_k}_i,\Sigma)-\sum_i L(\hat{\Sigma}_i,\Sigma) \right) / \sum_i L(\hat{\Sigma}^{LW_k}_i,\Sigma)
\end{equation}
 where the sum (over $i$) is over all datasets. For LW3, we replace  $\hat{\Sigma}^{LW_3}$ with  $\hat{\lambda}^{LW_3}$, $\hat{\Sigma}$ with $\hat{\lambda}$, and $\Sigma$ with $\lambda$ in (\ref{eq:PRIAL}) to compute the PRIAL.   Figure \ref{fig:PRIAL_LW1}, \ref{fig:PRIAL_LW2} and \ref{fig:PRIAL_LW3} are heatmaps of PRIAL with respect to LW1, LW2 and LW3 respectively. The rows correspond to various simulation scenarios (covariance structure, $p$ and $\gamma$) just described  and the columns correspond to various loss functions, described in Section \ref{sec:LossF}.  The red (blue) shades in the heatmap indicate positive (negative) PRIAL and the white represents zero PRIAL. The more intense the hue is,  the larger the absolute value of the PRIAL.  A positive PRIAL means our estimator compares favorably with LW1/LW2/LW3, negative PRIAL means the opposite and zero PRIAL means the estimators are comparable. Our estimator has positive PRIAL in 40.4\%, 45.7\% and 36.9\% of all simulation scenarios when compared to LW1, LW2 and LW3 respectively.

It seems  difficult to arrive at a general conclusion on which estimator (LW1, LW2,  LW3 or ours) is preferable for which covariance structure or which loss function although they perform similarly for the majority of scenarios.  Since LW1 is asymptotically optimal under the Frobenius loss function, one should expect LW1 to be  better than our estimator with the Frobenius loss. Figure \ref{fig:PRIAL_LW1} confirms this,  except  in the case in which the true covariance matrix is the identity ($\Sigma_5$) and in the case of $t$-distributed data with covariance matrix $\Sigma_8$  when our estimator has an advantage.  The  exception with the identity matrix is also expected because our approximate solution with $\kappa=0$ is a true maximum of the marginal likelihood function. For $p \geq 500$, LW1 outperforms our estimator when the true structure is spiked ($\Sigma_4$) under all loss functions and when the true structure represents a PCA-like situation ($\Sigma_2$) under $frob$, $Onenorm$ and $topev$ loss functions.  It is also expected that LW2, which is optimal under Stein's loss function, would outperform our estimator using Stein's loss. However, we observe in Figure \ref{fig:PRIAL_LW2} that, under the Stein's loss function, LW2 is better than our estimator only for the spiked covariance structure ($\Sigma_4$). On the other hand, our estimator tends to outperform LW2 with $\Sigma_8$ (robust case) and $\Sigma_5$ (identity) under multiple loss functions. If we compare estimators  with respect to the true covariance matrix over all loss functions, LW1  and LW2 are better than our estimator  for the spiked covariance structure ($\Sigma_4$), while our estimator is  better than  LW1 and LW2  when data are generated from a non-normal distribution with fat-tails ($\Sigma_8$)  and  it seems preferable to LW1 and  LW2, when the  true covariance matrix is  the identity ($\Sigma_5$).  When we compare our eigenvalue estimators with LW3 (Figure \ref{fig:PRIAL_LW3}), we see a similar pattern: the estimators perform comparably for the majority of scenarios, with LW3 outperforming our eigenvalue estimates for the spiked covariance structure ($\Sigma_4$) when $p \geq 500$ under all eigenvalue loss functions and for $\Sigma_2$ when $p \geq 500$ under the $evl2$ and $topev$.

\section{Linear Discriminant Analysis (LDA) on  Breast Cancer and Leukemia  Data} \label{sec:LDA}
In this section we apply our estimator to  two-class classification problems  using LDA in  breast cancer  and leukemia data. Specifically, we plug-in our estimator (and also LW1 and LW2) for the common covariance matrix of both classes in the discriminant function of LDA.
In the first application we consider  a  breast cancer  dataset. \cite{hessBCa} proposed a 31-probeset multigene predictor of pathologic complete response (pCR) to chemotherapy in an experiment with 133 patients with stage I-III breast cancer. Following \cite{hessBCa}, we split the samples into a training set of size 82 and a test set of size 51. We develop our classifier on a subset of the training data so that $n<p=31$ by randomly selecting 20 patients and preserving the ratio of the two classes.  We then evaluate and compare discrimination metrics, such as misclassification rate (MCR), sensitivity (Sens) and specificity (Spec), of the classifier that uses our estimator with those of the classifiers that use LW1 and LW2 as the plug-in estimators for the common-covariance matrix. The comparisons are presented in Table \ref{tab:BreastCancer}. To choose $\kappa$ for our estimator $\hat{\Sigma}^\kappa$, we follow the procedures described in Section \ref{sec:kappa} with nine loss functions and $\hat{\Sigma}^1_{NS}$ as the reference estimator for  the bootstrap-based method and the three loss functions for  the CV-based method. The combination of the $\kappa$-selection method and the corresponding loss function that was chosen is described in Table \ref{tab:BreastCancer}.  In addition, we also use LW1 and LW2 as the reference estimators in determining $\kappa$ for the bootstrap-based method. Our estimator has comparable but higher MCR and lower sensitivity and specificity compared to LW1 and LW2 when we use $\hat{\Sigma}^1_{NS}$ as the reference estimator to choose $\kappa$. However, when we use LW1 as the reference estimator our estimator improves MCR by 4\%, sensitivity by  7.7 \% and specificity by 2.6\% compared with LW1. Similarly, when using LW2 as the reference estimator in choosing $\kappa$, our estimator improves MCR by 15.7\% and specificity by 21\%. The estimator identified by the $\kappa$  that minimizes the criterion $\hat{q}_{\kappa}$ via cross-validation as described in Section \ref{sec:kappa} also performs well.

% Table generated by Excel2LaTeX from sheet 'Sheet1'
\begin{table}[htbp]
  \centering
  \caption{Error rates for LDA analysis of breast cancer data}
    \begin{tabular}{llrrr}
\toprule
    Estimator & Reference ($\bar{\Sigma}$) & \multicolumn{1}{l}{MCR} & \multicolumn{1}{l}{Sens} & \multicolumn{1}{l}{Spec} \\
\midrule
    LW1   & -    & 0.275 & 0.385 & 0.842 \\
    LW2   & -    & 0.333 & 0.538 & 0.711 \\
    $\tilde{\Sigma}$ (Boot-frob) & $\tilde{\lambda}^{NS}$ & 0.392 & 0.538 & 0.658 \\
    $\tilde{\Sigma}$ (Boot-st ) & LW1   & 0.235 & 0.462 & 0.868 \\
    $\tilde{\Sigma}$ (Boot-OneNorm) & LW2   & 0.176 & 0.538 & 0.921 \\
%   $\tilde{\Sigma}$ (CV-frob) & -  & 0.275 & 0.538 & 0.789 \\
   $\tilde{\Sigma}$ (CV-q) & -  & 0.216 & 0.538 & 0.868 \\
\bottomrule
    \end{tabular}%
  \label{tab:BreastCancer}%
\end{table}%

In the second application we  analyze   the  leukemia data set of \cite{golub}, consisting of gene expression measurements on 72 leukemia patients, 47 ALL  and  25 AML.   We  retain  $p=3571$ of the $7148$  gene expression levels for the analysis. This smaller  data set can be  downloaded from T. Hastie's  website \url{https://web.stanford.edu/~hastie/CASI_files/DATA/leukemia.html}.  A training set and a validation set   are then obtained by randomly selecting samples from the two classes but  preserving the proportion of the classes as in the original data, with the  training set   comprising  31 ALL and  17 AML patients.    We consider all LDA-classifiers as described in the breast cancer example,  which employ our estimator, LW1, and LW2. All these estimators achieve  perfect classification of the samples in the validation set.

\section{Summary and Conclusion}
Estimation of the covariance matrix is encountered in many statistical problems and has received much attention recently.  When $p$ is comparable to $n$ or even greater, the sample covariance matrix is a poor and ill-conditioned estimator primarily due to an overspread eigenspectrum. Several alternative estimators have been considered in the literature for such scenarios, some of which are asymptotically optimal with respect to certain loss functions and others are derived under strong structural assumptions on the covariance matrix (e.g.,  sparsity).   Often,  estimators are  valid in the regime in which  both $n$ and $p$ go to infinity in such a way that their ratio is finite.

In this paper, we consider the class of orthogonally equivariant estimators and propose an estimator that is valid  when $p>n$. This work is an extension of our previous work on equivariant estimation when $p<n$. Equivariance under orthogonal transformations reduces   the problem of estimating the covariance matrix to that of the estimation of its eigenvalues. To this end, we find a modification of the  profile likelihood function of eigenvalues by integrating out the sample eigenvectors. The integration result is approximate and valid for large $p$. The  critical  point of this pseudo-likelihood function,  a maximum under certain conditions, is in an estimator $\hat{\lambda}^0$ with all components equal,  thereby resulting in extreme shrinkage. To get distinct eigenvalue estimates, we perturb $\hat{\lambda}^0$ by introducing an approximate solution $\hat{\lambda}^1$ to the likelihood equations along with a tuning parameter $\kappa$. The tuning parameter, $\kappa \in [0,1)$, is selected by minimizing  the  risk, with respect to a loss function. We can find estimates of the risk using a bootstrap re-sampling scheme, which can be applied to any  problem with any loss function. The $\kappa$ selected with this method depends on the choice of a reference estimator, necessary to evaluate the loss function. Our estimator improves risk when a good estimator is employed as a reference estimator. Alternatively, a cross-validation estimate of the risk can be used, which was implemented for Frobenius, quadratic, and Stein's loss functions. We compare finite sample properties of our proposed estimator with  two  covariance matrix estimators of Ledoit and Wolf (Figure \ref{fig:PRIAL_LW1} and \ref{fig:PRIAL_LW2}) using Monte Carlo estimates of risk with respect to nine different loss functions and eight different covariance structures. We also compare the estimates of the population eigenvalues obtained by our method with those of  a consistent estimator for population eigenvalues proposed by Ledoit and Wolf (Figure \ref{fig:PRIAL_LW3}).  Furthermore, we  demonstrate in a real breast cancer example that our estimator can substantially improve risk.

The encouraging finite sample properties of our estimator reported here suggest that our method of constructing an orthogonally equivariant estimator on the marginal distribution of the sample eigenvalues may provide improved estimation of the covariance matrix, which is needed in many statistical applications.

\section*{Acknowledgments}
SB was supported by NIMH (National Institute of Mental Health) grant number P50 MH113838 for this work. Many thanks to Yiyuan Wu and Elizabeth Mauer for their assistance with the Figures.

\newpage
\section*{Appendix A}
In this appendix we  prove Proposition \ref{prop:eq:largep}. To do this, we first  review a result that is fundamental  in deriving the approximation.
Consider the integral
\[
I = \int_{O(p)} {\rm etr}( H X  H^{-1}  Y)(dH)
\]
where $(dH)$ is the   Haar measure of the group $O(p)$,   $X$ and $Y$  are $p \times p$  symmetric matrices with eigenvalues $x=(x_1, \ldots, x_p)$ and $y=(y_1, \ldots, y_p)$ respectively.   In  \cite{hikami-brezin}, this integral and some of its generalizations (to the unitary and symplectic groups),   known often as   Harish-Chandra-Itzykson-Zuber integrals, are   expressed in a form that involves the variables $\tau_{ij}=(x_i-x_j)(y_i-y_j)$, using the expansion about the saddle points of the integrand.  Specifically, for the orthogonal group,
\begin{equation}\label{eq:start}
 I={\rm e}^{\sum_{j=1}^p x_jy_j } \cdot f(\tau),
\end{equation}
where  the matrix $\tau$  is the symmetric matrix  with entries $\tau_{ij}$. The function  $f$  can be  expanded as a power series of $\tau$
\[
f=1+f_1+ \ldots+  f_r +O(\tau^{r+1})
\]
with the term $f_s$ being of order $s$ in $\tau$.    The solution (\ref{eq:start})  is obtained  from  a differential equation obeyed by the integral $I$.  The power series in $\tau$  for the integral $I$   is also  obtained from the expression of $I$ in terms of zonal polynomials. The term $f_s$ of the expansion  of $f$ is a  polynomial of degree $s$  in the $\tau$ variables. Each monomial  enjoys a graphical representation, as a graph that has $p$ nodes   and  an edge that connects  two nodes  $i,j$  if the  variable $\tau_{ij}$ appears in the monomial. In particular,   $\tau_{ij}^q$  is represented as  an edge of multiplicity $q$, that is,  as $q$ lines between the  nodes $i,j$. The total number of lines in such graph, and thus  the sum of the multiplicities of  all the edges, is $s$. The details can be found in \cite{hikami-brezin}, but a few example can make the discussion clearer.  In the term of order 1,
\[
f_1(\tau)=-\frac{1}{p}\sum_{i<j}\tau_{ij},
\]
each monomial $\tau_{ij}$ is represented by a graph  in which  two distinct nodes $i, j$ are joined by a simple (i.e., with multiplicity 1) edge and the remaining  nodes are singletons. Terms of  order $2$
\[
f_2(\tau) = \frac{3}{2p(p+2)}\sum_{i<j}\tau_{ij}^2 +\frac{1}{p(2+p)}\sum_{i,j,k}\tau_{ij}\tau_{jk}+\frac{p+1}{(p-1)p(p+2)} \sum_{i,j,k,l}\tau_{ij}\tau_{lk},
\]
are represented  by three (types of) graphs with two edges: the first sum is associated with  graphs  that consist of $p-2$ singletons and two nodes  joined by an edge of multiplicity two;  the second sum with graphs with $p-3$ singletons  and a connected component with 3 nodes and two simple  edges; the third sum  with  graphs  with $p-4$ singletons and two connected components, each having  two nodes and one simple edge.   The coefficients  $c_s$
of  the  monomials of  degree $s$ in $\tau$ have different forms in general. However, for large $p$  all coefficients are of the same form \citep{hikami-brezin}. Namely,
\[
c_s=(-1)^s\frac{g}{\prod_{t=0}^{s-1} (p+2t)} \left(1+O\left(\frac{2}{p}\right)\right),
\]
where $g$ is the degeneracy factor  due to  the multiplicities of the edges in the graph. For a graph with $q$ edges with multiplicities $(q_1, \ldots, q_q)$, the degeneracy is
\[
g=\prod_{i=1}^q \frac{1}{q_i!}\prod_{m=1}^{q_i-1}(1+2m)=\prod_{i=1}^q {2q_i \choose q_i}\frac{1}{2^{q_i}}.
\]
For example, the coefficients  $c_2$ for the monomials of degree 2 in  $f_2(\tau)$ above are
in the large $p$ limit $3/2p^2$ for the first  and $1/p^2$ for the other two. Because of this formula, one can  obtain the following large $p$ expression for $f$
\[
\prod_{i<j}\left( 1+\frac{2}{p}\tau_{ij} \right)^{-1/2}.
\]
The details of its derivation are omitted in  \cite{hikami-brezin}. We present them here. To make the notation more compact, let  us use a  bold symbol, ${\bf a}$, say,  for an unordered pair of integers $(i_a, j_a)$. In other words, we are using a vector notation for the  $P=p(p-1)/2$ distinct elements of  $\tau$, whose diagonal elements are zero. By applying the Taylor expansion of the inverse of the square root,
we find
\begin{eqnarray}\label{eq:taylor}
\prod_{i<j}\left( 1+\frac{2}{p}\tau_{ij} \right)^{-1/2}&=& \prod_{{\bf a}=1}^P\left( 1+\frac{2}{p}\tau_{\bf a} \right)^{-1/2} \nonumber\\
&=& \prod_{{\bf a}=1}^P\sum_{k=0}^\infty (-1)^k{2k \choose k}\frac{1}{2^k}\left(\frac{\tau_{\bf a}}{p}\right)^k\nonumber\\
&=&\sum_{s=0}^\infty \frac{(-1)^s}{p^s}\frac{1}{2^s} \sum_{\bf s}{2s_1 \choose s_1} \cdots {2s_P \choose s_P} \tau_{\bf 1}^{s_1}\cdots \tau_{\bf P}^{s_P},
\end{eqnarray}
where the last sum is over the vectors ${\bf s}=(s_1, \ldots, s_P)$ such that $0 \leq s_i \leq s$, with $\sum_{i=1}^Ps_i=s$.
Noticing that  singletons in the graph just contribute 1, one now recognizes immediately  that  in (\ref{eq:taylor}) for each $s=0, 1, \ldots$, the sum is over all possible graphs with $p$ nodes and with at most $s$  edges whose multiplicities sum to $s$. These are exactly the graphs associated with  the expansion of $f$.
Since the Taylor series  we have employed is convergent  when $-1 < 2\tau_{ij}/p<1$, the above result assumes
$0 \leq  2\tau_{ij}<p$ (since $\tau_{ij}$ is non-negative).  We can thus write (\ref{eq:start}) for large $p$  as
\begin{equation}\label{eq:asympt:hciz}
I\approx {\rm e}^{\sum_{i=1}^p x_iy_i } \cdot \prod_{i<j}^p\left( 1+ \frac{2}{p} \tau_{ij}\right)^{-1/2}.
\end{equation}
We use the approximation  (\ref{eq:asympt:hciz}) in the following
\begin{proof}[Proof of Proposition   \ref{prop:eq:largep}]
We follow    some steps as in Theorem 9.5.4 of  \cite{muirhead}. Consider first  the integral
\begin{equation*}
I(\ell)=\int_{ O(p)} {\rm etr}\left(-\frac{1}{2} \tilde{L} H^\top\Sigma^{-1} H\right) \, (H^\top dH)
\end{equation*}
where $H\in O(p)$,    $(H^\top dH)$  is the  Haar measure in $O(p)$ (whose integral gives the volume of $O(p)$) and $\tilde{L}={\rm diag}(\ell_1, \ldots, \ell_n, \ell, \ldots, \ell)$ with   $\ell$  repeated  $(p-n)$ times.  Partition $H=[H_1:H_2]$, where  $H_1 \in V_n({\mathbb R}^p)$ and thus $H_2$  is a matrix of order $p \times (p-n)$ whose columns are orthonormal and orthogonal to those of $H_1$. Since $H_2H_2^\top  =I_p -H_1H_1^\top$, we have
\begin{eqnarray*}
{\rm tr} (\tilde{L} H^\top \Sigma^{-1} H) &=& {\rm tr} \left (L H_1^\top \Sigma^{-1} H_1\right) + \ell \cdot {\rm tr}\left(H_2^\top \Sigma^{-1} H_2\right)\\
&=& {\rm  tr} \left( (L-\ell\,I_n) H_1^\top \Sigma^{-1} H_1\right) + \ell \, {\rm tr}\Sigma^{-1}
\end{eqnarray*}
where $L={\rm diag} (\ell_1, \ldots, \ell_n)$.
Using lemma 9.5.3 in \cite{muirhead}, it follows
\[
I(\ell)= {\rm Vol} \left(O(p-n)\right)\, {\rm etr} \left(-\frac{1}{2}\, \ell \cdot \Sigma^{-1} \right) \cdot J_n(\ell)
\]
where ${\rm Vol}\,\left(O(k)\right) = \frac{ 2^k \pi^{k^2/2}}{ \Gamma_{k} ( \frac{1}{2}k )}$  and
\[
J_n(\ell)=\int_{V_n({\mathbb R}^p)} {\rm etr}\left(-\frac{1}{2}(L-\ell\cdot I_n) H_1^\top \Sigma^{-1} H_1\right) (H_1^\top dH_1).
\]

Since the integrand function is bounded and $V_n({\mathbb R}^p)$ is compact,  in the limit $\ell\to 0$, we recover the integral of interest
\[
J_n= \lim_{\ell \to 0} J_n(\ell) = \frac{1}{{\rm Vol}\,\left(O(p-n)\right)} \, \lim_{\ell \to 0}I(\ell)
\]
with $\ell_1>\ell_2> \ldots > \ell_n>0$.
Applying the large-$p$  form (\ref{eq:asympt:hciz}) of the integral, we obtain
\[
I(\ell) \approx {\rm Vol}\,\left(O(p)\right) \cdot \exp\left( -\frac{1}{2}\sum_{i=1}^p \frac{ \tilde{\ell}_i}{\lambda_i} \right) \cdot \prod_{i<j}^p \left( 1 + \frac{1}{p} \left(\frac{\lambda_i -\lambda_j}{\lambda_i \lambda_j}\right)(\tilde{\ell}_i - \tilde{\ell}_j)\right)^{-1/2}
\]
and hence the result,  ignoring the ratio   $\frac{{\rm Vol}\,\left(O(p)\right)}{{\rm Vol}\,\left(O(p-n)\right)}={\rm Vol}\,\left( V_n({\mathbb R}^p)\right)$.
\end{proof}

\section*{Appendix B}
\begin{proof}[Proof of Proposition \ref{prop:properties}  (The  properties of the eigenvalue estimates)]
\begin{itemize}
\item[]
\item[a)]{
Let $a=1, \ldots, n$, and define $\psi_{a,b}= p+n(\ell_a-\ell_b)^2/\ell_a\ell_b $ and  $\ell_a-\ell_{a+1}=d_a$. Clearly,  $d_a > 0$  and $\psi_{ab} > p$. Then
\begin{eqnarray*}
n \hat{\lambda}^1_a &=& \ell_a-  \sum_{b=1}^n \frac{\ell_a - \ell_b}{  p +n \left(\frac{1}{\ell_b} -\frac{1}{\ell_a}\right)(\ell_a - \ell_b) }=\ell_{a+1}+d_a -\sum_{b=1}^n  \frac{\ell_a -\ell_{a+1} +\ell_{a+1}- \ell_b}{ \psi_{a,b}} \\
&>& \ell_{a+1}+d_a\left(1 -\frac{n}{p} \right)  -\sum_{b=1}^n  \frac{\ell_{a+1}- \ell_b}{ \psi_{a,b}} \\
& \geq & \ell_{a+1} -\sum_{b=1}^{a}  \frac{\ell_{a+1}- \ell_b}{ \psi_{a,b}} -\sum_{b=a+2}^n  \frac{\ell_{a+1}- \ell_b}{ \psi_{a,b}}  \\
& \geq & \ell_{a+1} -\sum_{b=1}^{a}  \frac{\ell_{a+1}- \ell_b}{ \psi_{a+1,b}} -\sum_{b=a+2}^n  \frac{\ell_{a+1}- \ell_b}{ \psi_{a,b}}  \\
& \geq & \ell_{a+1} -\sum_{b=1}^{n}  \frac{\ell_{a+1}- \ell_b}{ \psi_{a+1,b}} +\sum_{b=a+2}^{n}  \frac{\ell_{a+1}- \ell_b}{ \psi_{a+1,b}}       -\sum_{b=a+2}^n  \frac{\ell_{a+1}- \ell_b}{ \psi_{a,b}}  \\
&=& n \hat{\lambda}^1_{a+1} +\sum_{b=a+2}^{n} (\ell_{a+1}- \ell_b) \left( \frac{1} { \psi_{a+1,b}}-  \frac{1}{ \psi_{a,b}} \right) \\
&\geq & n \hat{\lambda}^1_{a+1}  >0
\end{eqnarray*}
where  we have used the fact that $\psi_{a+1,b} > \psi_{a,b}$ for $b \leq a$ and $\psi_{a,b} > \psi_{a+1,b}$ for $b \geq a +1$.   Hence  $\hat{\lambda}^{\kappa}_a  > \hat{\lambda}^{\kappa}_{a+1}$ for all $\kappa \in [0,1]$, and  $ \hat{\lambda}^{\kappa}_{n} > \hat{\lambda}^{\kappa}_{r}$, $r=n+1, \ldots, p$.}
\item[b)]{It follows immediately from the fact that $ \sum_{b,a=1}^n \frac{\ell_a - \ell_b}{  p +n \left(\frac{1}{\ell_b} -\frac{1}{\ell_a}\right)(\ell_a - \ell_b) }=0$.
}
\end{itemize}
\end{proof}

\section*{Appendix C}
In this appendix,  using standard arguments  \citep{muirhead},   we  compute the  profile log-likelihood for the eigenvalues $\lambda$ of the covariance matrix $\Sigma$.
Let $\Sigma=V \Lambda V^\top$ be the spectral decomposition of $\Sigma$.  The log-likelihood function  obtained from
(\ref{uhlig:density}) is
\begin{eqnarray*}
 {\mathcal L} (\Lambda, V) \equiv  {\mathcal L} (\Lambda, V |L,H_1) &=& -\frac{n}{2}\sum_{i=1}^p \ln \lambda_i  -\frac{1}{2}{\rm tr}\left(\Lambda^{-1} V^\top H_1 L  (V^\top H_1)^\top  \right)
\end{eqnarray*}
Since ${\rm tr}\left(\Lambda^{-1} A L  A^\top  \right) \geq   \sum_{i=1}^n\frac{\ell_i}{\lambda_i}$,  when $A$ satisfies the condition $A^\top A=I$,  with equality  when $A$ is one of the  $2^n$ matrices $M$ of dimensions $p \times n$  with components $M_{ij}=\pm \delta_{ij}$, where $\delta_{ij}$ is the Kronecker delta, we obtain
\[
 {\mathcal L} (\Lambda, V)  \leq  -\frac{n}{2}\sum_{i=1}^p \ln\lambda_i -\frac{1}{2}\sum_{i=1}^n\frac{\ell_i}{\lambda_i}
\]
and thus  the  expression on the right-hand side is  the  profile log-likelihood   ${\mathcal L}_P (\Lambda) ={\mathcal L} (\Lambda, \hat{V}_{\Lambda})$.
Thus the maximizer  $\hat{V}_{\Lambda}$  of the  log-likelihood for a fixed value of $\Lambda$ is a solution to $V^\top H_1=M$. Since $V$ is orthogonal,  then  $\hat{V}=[H_1M_n : H_2 P]$, where $M_n$ is  the $n \times n$ matrix of the first $n$ rows of $M$ and   $P  \in O(p-n)$ any orthogonal matrix.

\section*{Appendix D}
In this appendix we extend our algorithm to   the case in which  ${\rm rank}(S)=q \leq  n$, with  $q$  distinct positive eigenvalues.
In the rank-$q$  case, the density of $S$ and  the measure  are  obtained from those in the maximum rank case by replacing $n$  with $q$ \citep{DG_GJ_Mardia}. Namely,
\[
p(S)(dS) =K_q \left({\rm det} \Sigma \right)^{-q/2} {\rm etr}\left(-\Sigma^{-1}S/2 \right)\left({\rm det} L \right)^{(q-p-1)/2} (dS),
\]
where $S=H_1LH_1^\top$, with $L=\diag(\ell_1, \ldots, \ell_q)$,  and the volume form  written in terms of   the  Haar measure  $(H_1^\top dH_1)$  on $V_q({\mathbb R}^p)$ is
\[
(dS)= 2^{-q} \prod_{i=1}^q \ell_i^{p-q} \prod_{i<j}^q(\ell_i-\ell_j) (H_1^\top dH_1)  \wedge d\ell_1  \wedge \cdots  \wedge d\ell_q.
\]
Accordingly, all non-maximum rank formulae, from the marginal density of the
eigenvalues   to the ML  equations,    are  obtained by replacing $n$ with $q$ in the corresponding maximum rank equations. Thus,  an exact  solution to the  ML equations   gives all estimates   to be  $\sum_{a=1}^q\ell_a/pq$, and  our proposed estimates have the form
\begin{eqnarray*}
\hat{\lambda}^\kappa_a   & =& \frac{ \kappa}{q} \left(\ell_a -   \sum_{b=1}^q \frac{\ell_a - \ell_b}{  p +q \left(\frac{1}{\ell_b} -\frac{1}{\ell_a}\right)(\ell_a - \ell_b) } \right)+ (1- \kappa) \frac{ \sum_{b=1}^q  \ell_b}{pq}, \,  a=1, \ldots, q\\
 \hat{\lambda}^\kappa_r &=&  (1- \kappa) \frac{ \sum_{b=1}^q  \ell_b}{pq}, {\hskip 6 cm} r=q+1, \ldots, p
\end{eqnarray*}
with $0 \leq \kappa<1$. Such estimates can be shown to be positive and ordered  by following the same steps as in the proof of  Proposition \ref{prop:properties}.

\bibliographystyle{apalike}
\bibliography{CovRefsP}

\end{document}